\documentclass[11pt,reqno]{amsart}
\usepackage{enumerate}
\usepackage{amsfonts, amsmath, amssymb, amscd, amsthm, bm}
\def\qedbox{\hbox{$\rlap{$\sqcap$}\sqcup$}}
\def\qed{\nobreak\hfill\penalty250 \hbox{}\nobreak\hfill\qedbox}
\usepackage{enumerate}

\usepackage{multicol}
\usepackage{comment}

 \newtheorem{thm}{Theorem}[section]
 \newtheorem{cor}[thm]{Corollary}
  
 \newtheorem{lem}[thm]{Lemma}
 \newtheorem{prop}[thm]{Proposition}
 \theoremstyle{definition}
 
 \theoremstyle{remark}
 \newtheorem{rem}[thm]{Remark}

 \theoremstyle{claim}
 \numberwithin{equation}{section}

\def\l{\langle}
\def\r{\rangle}

\numberwithin{equation}{section}

\newcounter{rom}
\renewcommand{\therom}{(\roman{rom})}
{\end{list}}

\title{New characterizations of the Clifford torus as a Lagrangian self-shrinker}
\begin{document}

\author[H. Li]{Haizhong Li}
\address{Department of Mathematical Sciences,
Tsinghua University,  Beijing 100084, P. R. China}
\email{hli@math.tsinghua.edu.cn}

\author[X. Wang]{Xianfeng Wang}
\address{School of Mathematical Sciences and LPMC,
Nankai University,
Tianjin 300071,  P. R. China}
\email{wangxianfeng@nankai.edu.cn}

\thanks {The first author was supported in part by NSFC Grant No. 11271214.
The second author was supported in part by NSFC (Grant Nos. 11201243 and 11571185) and ``Specialized Research Fund for the Doctoral Program of Higher Education, Grant No. 20120031120026''.}



\begin{abstract}
In this paper, we obtain several new characterizations of the Clifford torus as a Lagrangian self-shrinker.  We first show that the Clifford torus $\mathbb{S}^1(1)\times\mathbb{S}^1(1)$ is the unique
compact orientable Lagrangian self-shrinker in $\mathbb{C}^2$ with $|A|^2\leq 2$, which gives an affirmative answer to Castro-Lerma's conjecture in \cite{Castro}.
We also prove that  the Clifford torus  is the unique compact  orientable embedded Lagrangian self-shrinker with nonnegative  or nonpositive Gauss curvature     in $\mathbb{C}^2$.
\end{abstract}

\maketitle

\section{Introduction}

Let $x:M^n\to \mathbb{R}^{n+p}$ be an $n$-dimensional submanifold in the $(n+p)$-dimensional Euclidean space. We call the immersed manifold $M^n$ a \textit{self-shrinker} if it satisfies the quasilinear elliptic system:
\begin{equation}\label{1.1}
\mathbf{H}=-x^{\bot},
\end{equation}
where $\mathbf{H}$ is the mean curvature vector and $\bot$ denotes the projection onto the normal bundle of $M^n$.

Self-shrinkers play an important role in the study of the mean curvature flow. Not only they correspond to self-shrinking solutions to the mean curvature flow,
 but also they describe all possible Type I blow ups at a given singularity of the mean curvature flow.
There are many results  about the classification of self-shrinkers.
In the curve case, Abresch and Langer \cite{AL} gave a complete classification of all solutions to \eqref{1.1}. These curves are called Abresch-Langer curves. In higher dimension and codimension one, Huisken (see \cite{Huis1990} and \cite{Huis1993}) proved
that $n$-dimensional smooth complete self-shrinkers  in $\mathbb{R}^{n+1}$ with $H\geq 0$, polynomial volume
growth, and $|A|$ bounded are $\Gamma\times\mathbb{R}^{n-1}$, or $\mathbb{S}^m(\sqrt{m})\times\mathbb{R}^{n-m}(0\leq m\leq n)$, where $\Gamma$ is an Abresch-Langer curve and $\mathbb{S}^m(\sqrt{m})$ is an $m$-dimensional sphere of radius $\sqrt{m}$. In \cite{CM}, Colding and Minicozzi  showed that Huisken's classification  holds without the assumption that $|A|$ is bounded.

In arbitrary codimensional case, Smoczyk \cite{Smo05} proved that (i) If $M^n$ is a compact self-shrinker in $\mathbb{R}^{n+p}$, then $M^n$ is a minimal submanifold of the
sphere $\mathbb{S}^{n+p-1}(\sqrt{n})$  if and only if $\mathbf{H}\neq 0$ and $\nabla^{\bot}\nu=0$, where $\nu=\mathbf{H}/|\mathbf{H}|$ is the principal normal. (ii)  Let $M^n$ be a complete non-compact self-shrinker in $\mathbb{R}^{n+p}$, if $\mathbf{H}\neq 0$, $\nabla^{\bot}\nu=0$, and $M^n$ has uniformly bounded geometry,  then $M^n$ is either $\Gamma\times\mathbb{R}^{n-1}$ or $N^m\times\mathbb{R}^{n-m}$, where $\Gamma$ is an Abresch-Langer curve and $N^m$ is an $m$-dimensional complete minimal submanifold  in $\mathbb{S}^{m+p-1}(\sqrt{m})$. In \cite{LiWei}, using the method of Colding and Minicozzi \cite{CM}, Li and Wei  showed that Smoczyk's result in complete non-compact case holds under a weaker condition.

We recall some  rigidity theorems for self-shrinkers. The first gap of the squared norm of the second fundamental form $|A|^2$ for self-shrinkers
was obtained by  Cao and Li \cite{CL} (which generalized codimension one case in \cite{LS}), they proved that
 if $M^n$ is a complete self-shrinker in $\mathbb{R}^{n+p}$, with polynomial volume growth and satisfying $|A|^2\leq 1$, then either $|A|^2=0$ and $M^n$ is  a hyperplane $\mathbb{R}^n$, or $|A|^2=1$ and $M^n$ is a round sphere $\mathbb{S}^n(\sqrt{n})$ or a cylinder $\mathbb{S}^m(\sqrt{m})\times\mathbb{R}^{n-m}(1\leq m\leq n-1)$. Cheng and Peng \cite{CP}  obtained some rigidity theorems on complete self-shrinkers without
assumption on polynomial volume growth. Ding and Xin \cite{DX} studied the second gap of $|A|^2$ for self-shrinkers in codimension one, they showed that if  $M^n$ is a complete self-shrinkers  in $\mathbb{R}^{n+1}$, with polynomial volume growth and satisfying $1\leq|A|^2\leq 1+0.022$, then $|A|^2=1$. Cheng and Wei \cite{CWei} improved the pinching constant 0.022 to 3/7 under the assumption that $|A|^2$ is constant.

In this paper, we are interested in rigidity results for compact Lagrangian self-shrinkers in $\mathbb{C}^2$.
An immersed manifold $M^n$ in  $\mathbb{C}^n$ is called a \textit{Lagrangian
submanifold} if the standard complex structure $J$ of $\mathbb{C}^n$
maps each tangent space of $M^n$ into its corresponding normal
space.
A Lagrangian submanifold $M^n$ in $\mathbb{C}^n$ is called a \textit{Lagrangian self-shrinker} if it satisfies \eqref{1.1}.
Recently, the study of Lagrangian self-shrinkers has drawn some attentions. For instance, many examples of Lagrangian self-shrinkers in $\mathbb{C}^n$ were constructed in \cite{Anciaux}, \cite{Castro2010} and \cite{LeeWang2010}, Hamiltonian stationary Lagrangian self-shrinkers in $\mathbb{C}^2$ were classified  in \cite{Castro2010}.
The canonical example of a compact Lagrangian self-shrinker in $\mathbb{C}^2$ is the Clifford torus $\mathbb{S}^1(1)\times\mathbb{S}^1(1)$, which is the standard example of monotone Lagrangian in  $\mathbb{C}^2$ (see \cite{N2010}).
In \cite{Castro}, Castro and Lerma obtained the following rigidity result for the Clifford torus.

\begin{thm}[see Theorem 1.2 in \cite{Castro}]\label{thm1.0}
Let $x:M^2\to \mathbb{C}^2$ be a compact orientable Lagrangian self-shrinker. If $|A|^2\leq 2$, then $|A|^2=2$ and  $M^2$ is a topological torus. If, in addition, the Gauss curvature $K$ of $M^2$ is nonnegative or nonpositive,  then $M^2$ is the Clifford torus $\mathbb{S}^1(1)\times\mathbb{S}^1(1)$.
\end{thm}

Castro and Lerma conjectured (see page 1519 in \cite{Castro}) that the condition ``the Gauss curvature $K$ of $M^2$ is nonnegative or nonpositive " is unnecessary in Theorem \ref{thm1.0}. Our following Theorem \ref{thm1.2} gives an affirmative answer to their conjecture. In fact, in Section 4, we prove
\begin{thm}\label{thm1.2}
Let $x:M^2\to \mathbb{C}^2$ be a compact orientable Lagrangian self-shrinker. If $|A|^2\leq 2$, then $|A|^2=2$ and $M^2$ is the Clifford torus $\mathbb{S}^1(1)\times\mathbb{S}^1(1)$.
\end{thm}

\begin{rem}\label{rem1.3}
For any $m,n\in\mathbb{N},~(m,n)=1,~m\leq n$,  Lee and Wang \cite{LeeWang2010} constructed the following example $T_{m,n}$ of Lagrangian self-shrinker by
$$\Psi_{m,n}:~\mathbb{R}^2\to\mathbb{C}^2,~(s,t)\longmapsto \sqrt{m+n}(\frac{1}{\sqrt{n}}\cos{s}~e^{i\sqrt{\frac{n}{m}}t},\frac{1}{\sqrt{m}}\sin{s}~e^{i\sqrt{\frac{m}{n}}t}),$$
with the squared norm of the second fundamental form satisfying $\frac{3m^2+n^2}{n(m+n)}\leq |A|^2\leq \frac{m^2+3n^2}{m(m+n)}$ (cf. \cite{Castro}).
$\forall \epsilon>0$, if $m$ and $n$ are integers satisfying that $m>\frac{3}{\epsilon}$ and $n=m+1$, then $(m,n)=1$ and $|A|^2\leq \frac{m^2+3n^2}{m(m+n)}< 2+\epsilon$, so there exist infinitely many examples $T_{m,n}$ satisfying $|A|^2\leq 2+\epsilon$.
 In other words, Lee-Wang's examples $T_{m,n}$ have an upper bound on $|A|^2$ which gets arbitrarily close to $2$.
This shows that the pinching constant $2$ is optimal in Theorem \ref{thm1.2}.
\end{rem}

In the last section, we prove the following classification theorem for compact orientable Lagrangian self-shrinker  in $\mathbb{C}^2$  with  nonnegative Gauss curvature.
\begin{thm}\label{thm1.3}
Let $x:M^2\to \mathbb{C}^2$ be a compact orientable Lagrangian self-shrinker. If the Gauss curvature $K$ of $M^2$ is nonnegative, then $K=0$ and $M^2$ is the Riemannian product of two closed Abresch-Langer curves.
\end{thm}

If $x:M^2\to \mathbb{C}^2$ is embedded, using the result of Abresch-Langer which states that
the only closed embedded self-shrinker in $\mathbb{R}^2$ is the circle, as an immediate consequence of Theorem \ref{thm1.3},  we obtain the following new characterization of  the Clifford torus $\mathbb{S}^1(1)\times\mathbb{S}^1(1)$.

\begin{cor}\label{cor1.4}
 The Clifford torus $\mathbb{S}^1(1)\times\mathbb{S}^1(1)$  is the unique compact orientable embedded  Lagrangian self-shrinker in $\mathbb{C}^2$ with nonnegative Gauss curvature.
\end{cor}

\begin{rem}
In \cite{N2011}, Neves proposed the following question (see Question 7.4 in \cite{N2011}): Find a condition on a Lagrangian torus in $\mathbb{C}^2$, which implies
that Lagrangian mean curvature flow $(L_t)_{0<t<T}$ will become extinct at time $T$ and, after rescale, $L_t$ converges to the Clifford torus. Our new characterizations of the Clifford torus might be useful to this question.
\end{rem}

\begin{rem}
We note that in Theorem \ref{thm1.2},  one does not need to assume that the Lagrangian self-shrinker is embedded, but we need this for Corollary \ref{cor1.4}.
We also note that the conclusion is still true if one replaces the assumption ``nonnegative Gauss curvature" by ``nonpositive Gauss curvature" in Corollary \ref{cor1.4}, see Corollary \ref{cor5.10} for more details.
\end{rem}

The paper is organized as follows: in the next section, we recall some basic formulas for Lagrangian submanifolds of $\mathbb{C}^2$.
In Section 3, we give some identities and lemmas for Lagrangian self-shrinkers of $\mathbb{C}^2$. In Section 4, we prove Theorem \ref{thm1.2}, we also prove that the Clifford torus $\mathbb{S}^1(1)\times\mathbb{S}^1(1)$ is the unique compact orientable Lagrangian self-shrinker in  $\mathbb{C}^2$  with $|A|^2$ being constant, which is the key step of the proof of Theorem \ref{thm1.2}. In Section 5, we prove Theorem \ref{thm1.3}.
Throughout this paper, we always assume that $M$ is connected and has no boundary.

\section{Preliminaries}
In this section,  $M^2$ will always denote a $2$-dimensional
Lagrangian submanifold of $\mathbb{C}^2$. We denote the Levi-Civita connections on $M^2$,
$\mathbb{C}^2$ and the normal bundle by $\nabla$, $D$ and
$\nabla^{\bot}$, respectively. The formulas of Gauss and Weingarten
are given by
\begin{equation}\label{GW}
D_XY=\nabla_XY+h(X,Y), ~D_X\xi=-A_{\xi}X+\nabla_X^{\bot}\xi,
\end{equation}
where $h$ is the second fundamental form, A denotes the shape operator, $X$ and $Y$ are tangent vector fields and $\xi$ is a normal
vector field on $M^2$.

The Lagrangian condition implies that (cf. \cite{LV}, \cite{LW})
\begin{equation}\label{kahler}\nabla_X^{\bot}JY=J\nabla_XY, ~A_{JX}Y=-J h(X,Y)=A_{JY}X.\end{equation}
The formulas above  imply that $\l h(X,Y),JZ\r$ is totally symmetric, i.e.,
\begin{equation}\label{2.1}
\l h(X,Y),JZ\r= \l h(Y,Z),JX\r=\l h(Z,X),JY\r,
\end{equation}
where $\l,\r$ denotes the standard inner product in $\mathbb{C}^2$.

For a Lagrangian submanifold $M^2$ in $\mathbb{C}^2$, an orthonormal frame field
$$e_1,e_2,e_{1^*},e_{2^*}$$
is called an \textit{adapted Lagrangian frame field} if $e_1,,e_2$ are orthonormal tangent vector fields and $e_{1^*},e_{2^*}$
are normal vector fields given by
\begin{equation}\label{2.2}
e_{1^*}=Je_1,e_{2^*}=Je_2.
\end{equation}
The dual frame fields of $e_1,e_2$ are $\theta_1,\theta_2$, the Levi-Civita connection forms and normal connection forms are $\theta_{ij}$ and
$\theta_{i^*j^*}$, respectively.
Writing $h(e_i,e_j)=\sum\limits_kh_{ij}^{k^*}e_{k^*}$, \eqref{2.1} is equivalent to
\begin{equation}\label{2.3}
h_{ij}^{k^*}=h_{jk}^{i^*}=h_{ki}^{j^*},~1\leq i,j,k\leq 2.
\end{equation}

We have the following structure equations.
\begin{equation}\label{str1}
dx=\sum_i\theta_ie_i,
\end{equation}
\begin{equation}\label{str2}
de_i=\sum_j\theta_{ij}e_j+\sum_{j,k}h_{ij}^{k^*}\theta_je_{k^*},
\end{equation}
\begin{equation}\label{str3}
de_{k^*}=-\sum_{i,j}h_{ij}^{k^*}\theta_je_{i}+\sum_l\theta_{k^*l^*}e_{l^*}.
\end{equation}

If we denote the components of curvature tensors of $\nabla$ and $\nabla^{\bot}$ by $R_{ijkl}$ and $R_{i^*j^*kl}$, respectively, then the equations of
Gauss, Codazzi and Ricci are given by  (cf. \cite{LV}, \cite{LW})
\begin{equation}\label{2.4}
R_{milp}=\sum_j(h_{ml}^{j^*}h_{ip}^{j^*}-h_{mp}^{j^*}h_{il}^{j^*}),
\end{equation}
\begin{equation}\label{2.4-2}
R_{jk}=\sum_{p}H^{p^*}h_{jk}^{p^*}-\sum_{i,p}h_{ij}^{p^*}h_{ik}^{p^*},
\end{equation}
\begin{equation}\label{2.5}
h_{ij,l}^{k^*}=h_{il,j}^{k^*},~1\leq i,j,k,l\leq 2,
\end{equation}
\begin{equation}\label{2.6}
R_{i^*j^*kl}=\sum_m(h_{mk}^{i^*}h_{ml}^{j^*}-h_{ml}^{i^*}h_{mk}^{j^*}),
\end{equation}
\begin{equation}\label{Gauss}
R=H^2-|A|^2,
\end{equation}
where $R_{jk}$ and $R$ are the Ricci curvature and the scalar curvature of $M^2$, respectively, $|A|^2=\sum\limits_{i,j,k}(h_{ij}^{k^*})^2$ is the squared norm  of the second fundamental form, $\mathbf{H}=\sum\limits_{k}H^{k^*}e_{k^*}=\sum\limits_{i,k}h_{ii}^{k^*}e_{k^*}$ is the mean curvature vector field, $H=|\mathbf{H}|$ is the
mean curvature of $M^2$, and $h_{ij,l}^{k^*}$ is defined by
\begin{equation}\label{2.7}
\sum_lh_{ij,l}^{k^*}\theta_l=dh_{ij}^{k^*}+\sum_lh_{lj}^{k^*}\theta_{li}+\sum_lh_{il}^{k^*}\theta_{lj}+\sum_mh_{ij}^{m^*}\theta_{m^*k^*}.
\end{equation}
We can write \eqref{2.7} in the following equivalent form:
\begin{equation}\label{2.8}
\begin{aligned}
(\nabla_X h)(Y,Z)=\nabla^{\bot}_X h(Y,Z)-h(\nabla_X
Y,Z)-h(Y,\nabla_X Z),
\end{aligned}
\end{equation}
where $X$, $Y$  and $Z$ are tangent vector fields on $M^2$. We note that $(\nabla_{e_k}h)(e_i,e_j)=\sum\limits_lh_{ij,k}^{l^*}e_{l^*}$.

Combining \eqref{2.3} and \eqref{2.5}, we know that $h_{ij,l}^{k^*}$ is totally symmetric, i.e.,
\begin{equation}\label{2.9}
h_{ij,l}^{k^*}=h_{jl,k}^{i^*}=h_{lk,i}^{j^*}=h_{ki,j}^{l^*},~1\leq i,j,k,l\leq 2.
\end{equation}

We  have the following Ricci identities.
\begin{equation}\label{2.10}
h_{ij,lp}^{k^*}-h_{ij,pl}^{k^*}=\sum_mh_{mj}^{k^*}R_{milp}+\sum_mh_{im}^{k^*}R_{mjlp}+\sum_mh_{ij}^{m^*}R_{m^*k^*lp},
\end{equation}
where $h_{ij,lp}^{k^*}$ is defined by
\begin{equation}\label{2.11}
\sum_ph_{ij,lp}^{k^*}\theta_p=dh_{ij,l}^{k^*}+\sum_ph_{pj,l}^{k^*}\theta_{pi}+\sum_ph_{ip,l}^{k^*}\theta_{pj}+\sum_ph_{ij,p}^{k^*}\theta_{pl}+
\sum_ph_{ij,l}^{p^*}\theta_{p^*k^*}.
\end{equation}

Using \eqref{2.3}, \eqref{2.4} and \eqref{2.6}, we have
\begin{equation}\label{2.12}
R_{m^*i^*lp}=R_{milp}.
\end{equation}

We define the first and second covariant derivatives, and Laplacian of the mean curvature vector field $\mathbf{H}=\sum\limits_{k}H^{k^*}e_{k^*}$ in the normal
bundle $N(M^2)$ as follows.
\begin{equation}\label{2.17}
\sum_iH_{,i}^{k^*}\theta_i=dH^{k^*}+\sum_lH^{l^*}\theta_{l^*k^*},
\end{equation}
\begin{equation}\label{2.18}
\sum_jH_{,ij}^{k^*}\theta_j=dH_{,i}^{k^*}+\sum_jH_{,j}^{k^*}\theta_{ji}+\sum_{l}H_{,i}^{l^*}\theta_{l^*k^*}.
\end{equation}
\begin{equation}\label{2.19}
\Delta ^{\bot}H^{k^*}=\sum_{i}H_{,ii}^{k^*},~H^{k^*}=\sum_ih_{ii}^{k^*}.
\end{equation}

Let $f$ be a smooth function on $M^2$, we define the covariant derivatives $f_{,i},f_{,ij}$, and the Laplacian of $f$ as follows.
\begin{equation}\label{2.20}
df=\sum_if_{,i}\theta_i,~\sum_jf_{,ij}\theta_j=df_{,i}+\sum_jf_{,j}\theta_{ji},~\Delta f=\sum_i f_{,ii}.
\end{equation}

\section{Some  Identities and Lemmas}
In this section, we assume that $x:M^2\to \mathbb{C}^2$ is a compact orientable Lagrangian self-shrinker.
The self-shrinker equation \eqref{1.1} is equivalent to
\begin{equation}\label{3.1}
H^{k^*}=-\l x,e_{k^*}\r,~1\leq k\leq 2.
\end{equation}

\begin{lem}[cf. \cite{CL}]\label{lem3.0} Let $x:M^2\to \mathbb{C}^2$ be a Lagrangian self-shrinker, we have
\begin{equation}\label{3.2}
H_{,i}^{k^*}=\sum_jh_{ij}^{k^*}\l x,e_{j}\r,~1\leq i,k\leq 2,
\end{equation}
\begin{equation}\label{3.3}
H_{,ij}^{k^*}=\sum_mh_{im,j}^{k^*}\l x,e_{m}\r+h_{ij}^{k^*}-\sum_{m,p}H^{p^*}h_{im}^{k^*}h_{mj}^{p^*},~1\leq i,j,k\leq 2.
\end{equation}
\end{lem}
\begin{proof}

From \eqref{2.20} and the structure equations \eqref{str1}-\eqref{str3}, we  obtain
\begin{equation}\label{xij}
x_{,i}=e_i,~x_{,ij}=e_{i,j}=\sum_kh_{ij}^{k^*}e_{k^*}.
\end{equation}
and
\begin{equation}\label{ekij}
e_{k^*,i}=-\sum_jh_{ij}^{k^*}e_j,~e_{k^*,ij}=-\sum_mh_{im,j}^{k^*}e_m-\sum_{m,p}h_{im}^{k^*}h_{mj}^{p^*}e_{p^*}.
\end{equation}

Taking covariant derivative of \eqref{3.1} with respect to $e_i$ by use of \eqref{xij} and \eqref{ekij}, we obtain \eqref{3.2}.
Taking covariant derivative of \eqref{3.2} with respect to $e_j$ by use of \eqref{xij} and \eqref{3.1}, we obtain \eqref{3.3}.
\end{proof}

Recall the following operator $\mathcal{L}$ which was introduced and studied firstly on self-shrinkers by Colding and Minicozzi (see (3.7) in \cite{CM}):
$
\mathcal{L}=\Delta-\l x, \nabla\cdot\r=e^{|x|^2/2}\text{div}{(e^{-|x|^2/2}\nabla\cdot)},
$
where $\Delta,\nabla$ and div denote the Laplacian, gradient and divergent operator on the self-shrinker, respectively. The operator $\mathcal{L}$ is self-adjoint in a weighted $L^2$ space.

\begin{lem}\label{lem3.1}
Let $x:M^2\to \mathbb{C}^2$ be a Lagrangian self-shrinker, we have
\begin{equation}\label{DeltaA}
\begin{aligned}
\frac{1}{2}\mathcal{L} |A|^2=|\nabla A|^2+|A|^2-\frac{3}{2}|A|^4+2H^2|A|^2-\frac{1}{2}H^4-\sum_{i,j,k,l}H^{k^*}H^{l^*}h_{ij}^{k^*}h_{ij}^{l^*}.
\end{aligned}
\end{equation}
\end{lem}

\begin{proof}
By definition of $\Delta$ and using \eqref{2.4},\eqref{2.6},\eqref{2.9},\eqref{2.10}, \eqref{2.12} and \eqref{3.3}, we have
\begin{equation}\label{deltaA1}
\begin{aligned}
&\frac{1}{2}\Delta |A|^2=\sum_{i,j,k,p}(h_{ij,k}^{p^*})^2+\sum_{i,j,k,p}h_{ij}^{p^*}h_{ij,kk}^{p^*}\\
=&|\nabla A|^2+\sum_{i,j,k,p}h_{ij}^{p^*}h_{kk,ij}^{p^*}+\sum_{i,j,m,p}h_{ij}^{p^*}h_{mi}^{p^*}R_{mj}+\sum_{i,j,k,m,p}h_{ij}^{p^*}h_{km}^{p^*}R_{mijk}
+\sum_{i,j,k,m,p}h_{ij}^{p^*}h_{ki}^{m^*}R_{m^*p^*jk}\\
=&|\nabla A|^2+\sum_k\frac{1}{2}(|A|^2)_{,k}\l x,e_k\r+|A|^2-\sum_{i,j,k,m,p}H^{m^*}h_{ij}^{p^*}h_{ik}^{p^*}h_{kj}^{m^*}\\
&+\sum_{i,j,m,p}h_{ij}^{p^*}h_{mi}^{p^*}R_{mj}+\sum_{i,j,k,m,p}h_{ij}^{p^*}h_{km}^{p^*}R_{mijk}
+\sum_{i,j,k,m,p}h_{ij}^{p^*}h_{ki}^{m^*}R_{m^*p^*jk}.
\end{aligned}
\end{equation}
Since $M^2$ is a Lagrangian surface in $\mathbb{C}^2$, denote the Gauss curvature of $M^2$  by $K$, from \eqref{2.4} and \eqref{2.12}, we have
\begin{equation}\label{RK}
\begin{aligned}
&-\sum_{i,p}h_{ij}^{p^*}h_{ik}^{p^*}=K\delta_{jk}-\sum_{p}H^{p^*}h_{jk}^{p^*},\\
&R_{mijk}=K(\delta_{mj}\delta_{ik}-\delta_{mk}\delta_{ij}),~R_{m^*p^*jk}=K(\delta_{mj}\delta_{pk}-\delta_{mk}\delta_{pj}),
\end{aligned}
\end{equation}
 substituting \eqref{RK} into \eqref{deltaA1}, using Gauss equation \eqref{Gauss}, we obtain
\begin{equation}
\begin{aligned}
\frac{1}{2}\mathcal{L} |A|^2=|\nabla A|^2+|A|^2-\frac{3}{2}|A|^4+2H^2|A|^2-\frac{1}{2}H^4-\sum_{i,j,k,l}H^{k^*}H^{l^*}h_{ij}^{k^*}h_{ij}^{l^*}.
\end{aligned}
\end{equation}
\end{proof}

\begin{lem}[cf. \cite{CL}, \cite{CM}]\label{lem3.2}
Let $x:M^2\to \mathbb{C}^2$ be a compact orientable Lagrangian self-shrinker, we have
\begin{equation}\label{lx2}
\begin{aligned}
&0=\int_M^2\frac{1}{2}\Delta (|x|^2)dv=\int_M^2(2-H^2)dv,\\
&0=\int_M^2\frac{1}{2}\mathcal{L}(|x|^2)e^{-\frac{|x|^2}{2}}dv=\int_M^2(2-|x|^2)e^{-\frac{|x|^2}{2}}dv.
\end{aligned}
\end{equation}
\end{lem}

\begin{proof}
It follows from \eqref{xij} and \eqref{3.1} that
$$\frac{1}{2}\Delta (|x|^2)=2+\l x,\Delta x\r=2+\sum_kH^{k^*}\l x,e_{k^*}\r=2-H^2,~\text{which implies that}$$
$0=\int_M^2\frac{1}{2}\Delta (|x|^2)dv=\int_M^2(2-H^2)dv$ and $\frac{1}{2}\mathcal{L}(|x|^2)=\frac{1}{2}\Delta (|x|^2)-\sum\limits_i\l x,e_i\r^2=2-|x|^2$.
The last equation in \eqref{lx2} follows from the fact that the operator $\mathcal{L}$ is self-adjoint in a weighted $L^2$ space.
\end{proof}

\section{Proof of Theorem \ref{thm1.2}}
In this section, we prove Theorem \ref{thm1.2}. First, we recall the following lemma which is important in the proof of our key Proposition \ref{prop3.5}. It was proved in \cite{Castro} by Castro and Lerma by using Gauss-Bonnet theorem combined with the Gauss equation in a clever way.

\begin{lem}[see Theorem 1.2 in \cite{Castro}]\label{lem3.3}
Let $x:M^2\to \mathbb{C}^2$ be a compact orientable Lagrangian self-shrinker. If $|A|^2\leq 2$ , then $|A|^2=2$ and $M^2$ is a topological torus.
\end{lem}
\begin{proof}
Denote the Gauss curvature of $M^2$ by $K$. From Gauss equation $R=2K=H^2-|A|^2$ and Gauss-Bonnet theorem,  we have
\begin{equation}8\pi(1-\text{gen}(M^2))=2\int_M^2K dv=\int_M^2(H^2-|A|^2)dv=\int_M^2(2-|A|^2)dv,\end{equation}
where $\text{gen}(M^2)$ stands for the genus of $M^2$ and the last equality is due to Lemma \ref{lem3.2}.
It is well known that there exist no Lagrangian self-shrinkers in $\mathbb{C}^n$ with the topology of sphere, which
was proved by Smoczyk (see \cite{Smo00}, Theorem 2.3.5, see also Theorem 2.1 in \cite{Castro} for a detailed proof).  Hence, if  $|A|^2\leq 2$ , then $|A|^2=2$ and $M^2$ is a topological  torus.
\end{proof}

We are now ready to prove the following key proposition:
\begin{prop}\label{prop3.5}
Let $x:M^2\to \mathbb{C}^2$ be a compact orientable Lagrangian self-shrinker. If the squared norm of the second fundamental form $|A|^2$ is constant, then  $|A|^2=2$ and
$M^2$ is the Clifford torus $\mathbb{S}^1(1)\times\mathbb{S}^1(1)$.
\end{prop}
\begin{proof}
We prove by two steps.
Firstly, we show that  $|A|^2=2$.

Since $M^2$ is compact, there exists a point $p_0\in M^2$ such that  $|x|^2$ attains its minimum at $p_0$.
We immediately have $(|x|^2)_{,j}=0,~1\leq j\leq 2$ at $p_0$, which implies that $\l x,e_j\r(p_0)=0,~1\leq j\leq 2$.
Hence at $p_0$, from \eqref{3.1} and \eqref{3.2} we have $x=-\mathbf{H},~|x|^2=H^2$, $H^{k^*}_{,i}=0,~1\leq i,k\leq 2$,
which lead to the following equations:
\begin{equation}\label{3.7}
h_{11,1}^{1^*}+h_{22,1}^{1^*}=0,~h_{11,2}^{1^*}+h_{22,2}^{1^*}=0,~h_{11,2}^{2^*}+h_{22,2}^{2^*}=0.
\end{equation}
On the other hand, since $|A|^2=(h_{11}^{1^*})^2+3(h_{12}^{1^*})^2+3(h_{12}^{2^*})^2+(h_{22}^{2^*})^2$ is constant, we have $(|A|^2)_{,k}=0,~1\leq k\leq 2$. Therefore,
\begin{equation}\label{3.8}
\begin{aligned}
&h_{11}^{1^*}h_{11,1}^{1^*}+3h_{12}^{1^*}h_{12,1}^{1^*}+3h_{12}^{2^*}h_{12,1}^{2^*}+h_{22}^{2^*}h_{22,1}^{2^*}=0,\\
&h_{11}^{1^*}h_{11,2}^{1^*}+3h_{12}^{1^*}h_{12,2}^{1^*}+3h_{12}^{2^*}h_{12,2}^{2^*}+h_{22}^{2^*}h_{22,2}^{2^*}=0.
\end{aligned}
\end{equation}
From \eqref{3.7}, using \eqref{2.9}, we get
\begin{equation}\label{n38}
h_{22,1}^{1^*}=-h_{11,1}^{1^*},~h_{22,2}^{1^*}=-h_{11,2}^{1^*},~h_{22,2}^{2^*}=h_{11,1}^{1^*}.
\end{equation}
Since  $h_{ij}^{k^*}$  and $h_{ij,l}^{k^*}$ are both totally symmetric (see \eqref{2.3} and \eqref{2.9}), by substituting \eqref{n38} into \eqref{3.8}, we obtain
\begin{equation}\label{square1}
(h_{11}^{1^*}-3h_{22}^{1^*}) h_{11,1}^{1^*} -(h_{22}^{2^*} -3h_{11}^{2^*})h_{11,2}^{1^*} =0,
\end{equation}
\begin{equation}\label{square2}
(h_{22}^{2^*} -3h_{11}^{2^*})h_{11,1}^{1^*}+(h_{11}^{1^*}-3h_{22}^{1^*}) h_{11,2}^{1^*}  =0.
\end{equation}
Taking the sum of the square of \eqref{square1} and the square of \eqref{square2}, we get
\begin{equation}\label{square}
[(h_{11}^{1^*}-3h_{22}^{1^*})^2+(h_{22}^{2^*} -3h_{11}^{2^*})^2][(h_{11,1}^{1^*})^2+ (h_{11,2}^{1^*})^2]=0.
\end{equation}
Hence, from \eqref{square}, we have the following two possibilities:

(i) At $p_0$, $h_{11}^{1^*}=3h_{22}^{1^*},~h_{22}^{2^*}=3h_{11}^{2^*}$. In this case, $|A|^2(p_0)=\frac{4}{3}((h_{11}^{1^*})^2+(h_{22}^{2^*})^2),H^2(p_0)=\frac{16}{9}((h_{11}^{1^*})^2+(h_{22}^{2^*})^2)$. From Lemma \ref{lem3.2}, as $|x|^2$ attains its minimum at $p_0$,
we obtain that
$H^2(p_0)=|x|^2(p_0)\leq 2$, hence $|A|^2=|A|^2(p_0)=\frac{3}{4}H^2(p_0)\leq\frac{3}{2}$, which is impossible by Lemma \ref{lem3.3}.

(ii) At $p_0$,  $h_{11,1}^{1^*}=h_{11,2}^{1^*}=0$. Then from \eqref{n38}, we have $h_{22,1}^{1^*}=h_{22,2}^{1^*}=h_{22,2}^{2^*}=0$.
In this case, $|\nabla A|=0$ at $p_0$. By Lemma \ref{lem3.1}, $|A|^2-\frac{3}{2}|A|^4+2H^2|A|^2-\frac{1}{2}H^4-\sum_{i,j,k,l}H^{k^*}H^{l^*}h_{ij}^{k^*}h_{ij}^{l^*}=0$ at $p_0$.
If $H(p_0)=0$, then we have $|A|^2-\frac{3}{2}|A|^4=0$, which implies that $|A|^2=0$ or $|A|^2=\frac{2}{3}$. By Lemma \ref{lem3.3}, this can not occur, so we have  $H(p_0)\neq0.$
Since at $p_0$, $\mathbf{H}\neq 0$, we choose local orthonormal frame $\{e_1,e_2\}$ such that $e_1//J\mathbf{H}$ and $H^{1^*}=H,~H^{2^*}=0$, then at $p_0$, we have
\begin{equation}\label{n319}
\begin{aligned}
|A|^2-\frac{1}{2}|A|^4&=|A|^4-2H^2|A|^2+\frac{1}{2}H^4+\sum_{i,j}H^2(h_{ij}^{1^*})^2\\
&=(|A|^2-H^2)^2+H^2\sum_{i,j}(h_{ij}^{1^*}-\frac{1}{2}H\delta_{ij})^2\geq 0,\\
\end{aligned}
\end{equation}
which implies that
$|A|^2=|A|^2(p_0)\leq 2$. Using Lemma \ref{lem3.3},
we know that $|A|^2\equiv 2$.

Secondly, we show that $M^2$ is the Clifford torus $\mathbb{S}^1(1)\times\mathbb{S}^1(1)$.

From the  arguments above, we know that  $|\nabla A|=0$ at $p_0$. Moreover, since $|A|^2=2=\frac{1}{2}|A|^4$,
from \eqref{n319}, we know that $(|A|^2-H^2)^2+H^2\sum_{i,j}(h_{ij}^{1^*}-\frac{1}{2}H\delta_{ij})^2=0$ at $p_0$, which immediately implies that $|A|^2=H^2$ and $h_{ij}^{1^*}=\frac{1}{2}H\delta_{ij}$ at $p_0$,
we also know that $H^{1^*}=h_{11}^{1^*}+h_{22}^{1^*}=H,~H^{2^*}=h_{11}^{2^*}+h_{22}^{2^*}=0$ at $p_0$, hence we get that
$h_{11}^{1^*}=h_{22}^{1^*}=\frac{1}{2}H$ and $h_{11}^{2^*}=h_{22}^{2^*}=0$ at $p_0$.
Therefore, we have $2=|A|^2=|A|^2(p_0)=H^2(p_0)=|x|^2(p_0)$. Since $|x|^2$ attains its minimum at
$p_0$, we get $|x|^2\geq 2$, which together with Lemma \ref{lem3.2} imply that $|x|^2\equiv2$.

Since $|x|^2\equiv2$, we have that $\langle x,e_i\rangle=0,~i=1,2$, which means that the position vector $x$ is equal to $x^{\bot}$. Using the self-shrinker equation \eqref{1.1}, we immediately have that
$\mathbf{H}=-x$ and $H^2=|x|^2\equiv2$. In particular, we get that $H=|\mathbf{H}|=\sqrt{2}$ and from \eqref{GW} we have $$\nabla_{e_i}^{\bot}\mathbf{H}=D_{e_i}\mathbf{H}+A_{\mathbf{H}}e_i=D_{e_i}(-x)+A_{\mathbf{H}}e_i=-e_i+A_{\mathbf{H}}e_i,~i=1,2,$$
where in the last equality we use the fact that $x$ is the position vector. In the  equation above, $\nabla_{e_i}^{\bot}\mathbf{H}$ is a normal vector, $-e_i+A_{\mathbf{H}}e_i$ is a tangent vector, we get that both of them have to vanish, so we obtain that $\mathbf{H}$ is a non-null parallel normal vector field and hence $J\mathbf{H}$ is a non-null parallel tangent vector field on $M^2$. We have also shown that  $|A|^2=2$.
It follows that $H^{k^*}_{,i}=0,~ (|A|^2)_{,k}=0,~1\leq i,k\leq 2$, which means that
both \eqref{3.7} and \eqref{3.8}  hold at $\forall~ p\in M^2$. Thus,  $\forall~ p\in M^2$, using an analogous argument to that in the first step of the proof,
there are two possibilities.
(i) At $p$, $h_{11}^{1^*}=3h_{22}^{1^*},~h_{22}^{2^*}=3h_{11}^{2^*}$. In this case, $|A|^2=\frac{3}{4}H^2$, which is a contradiction with $|A|^2=H^2=2$.
(ii) At $p$,  $h_{11,1}^{1^*}=h_{11,2}^{1^*}=h_{22,1}^{1^*}=h_{22,2}^{1^*}=h_{22,2}^{2^*}=0$.
Hence, we obtain that $|\nabla A|=0$,~$\forall~ p\in M^2.$

Since $H=\sqrt{2}\neq 0$, we  choose local orthonormal frame $\{e_1,e_2\}$ such that $e_1//J\mathbf{H}$ and $H^{1^*}=H,~H^{2^*}=0$.
As $|\nabla A|=0$, we get  that  \eqref{n319} holds at $\forall~ p\in M^2$.  As $|A|^2=2=\frac{1}{2}|A|^4$,
from \eqref{n319}, we know that $(|A|^2-H^2)^2+H^2\sum_{i,j}(h_{ij}^{1^*}-\frac{1}{2}H\delta_{ij})^2=0$, we also know that $H^{1^*}=H,~H^{2^*}=0$, hence  under the orthonormal frame $\{e_1,e_2\}$ chosen above, we have
  $h_{11}^{1^*}=\frac{\sqrt{2}}{2},~h_{22}^{1^*}=\frac{\sqrt{2}}{2},~h_{11}^{2^*}=0,~h_{22}^{2^*}=0$.

In the following, we will determine the explicit expression of the immersion $x$, up to an isometry of $\mathbb{C}^2$.
Since $H=\sqrt{2}$ is constant and $J\mathbf{H}$ is a non-null parallel tangent vector field on $M^2$, we get that
$e_1$ is parallel on $M^2$, hence $\nabla_{e_i}e_j=0,i,j=1,2.$  Therefore, there exist local coordinates $\{u,v\}$
such that $e_1=\frac{\partial}{\partial u},e_2=\frac{\partial}{\partial v}$.
 Since $e_1=\frac{\partial}{\partial u}$ and $e_2=\frac{\partial}{\partial v}$ are orthonormal, $x$ is a Lagrangian immersion, we get
\begin{equation}\label{x1}
\l x_u,x_u\r=\l x_v,x_v\r=1,~\l x_u,x_v\r=\l x_u,i x_v\r=0.
\end{equation}
From $\nabla_{e_i}e_j=0,i,j=1,2$ and  $h_{11}^{1^*}=\frac{\sqrt{2}}{2},~h_{22}^{1^*}=\frac{\sqrt{2}}{2},~h_{11}^{2^*}=0,~h_{22}^{2^*}=0$,
we have
\begin{equation}\label{x2}
\begin{aligned}
x_{uu}=x_{vv}=\frac{\sqrt{2} i}{2}x_u,
~x_{uv}=\frac{\sqrt{2} i}{2}x_v.
\end{aligned}
\end{equation}
The self-shrinker condition \eqref{1.1} and $|x|^2=2$ imply that
 \begin{equation}\label{x3}
\mathbf{H}=x_{uu}+x_{vv}=-x, ~\l x,x\r=2.
\end{equation}
Using \eqref{x1}-\eqref{x3}, we obtain the following explicit expression of $x$.
\begin{equation}\label{x4}
x(u,v)=e^{\frac{iu}{\sqrt{2}}}(a_1e^{\frac{iv}{\sqrt{2}}}+a_2e^{\frac{-iv}{\sqrt{2}}},b_1e^{\frac{iv}{\sqrt{2}}}+b_2e^{\frac{-iv}{\sqrt{2}}})\in\mathbb{C}^2,
\end{equation}
where $a_1,a_2,b_1,b_2$ are constant complex numbers satisfying that
$a_1\bar{a_1}+b_1\bar{b_1}=a_2\bar{a_2}+b_2\bar{b_2}=1,a_1\bar{a_2}+b_1\bar{b_2}=0$. Therefore, up to an isometry of $\mathbb{C}^2$, $x$ is  congruent with
\begin{equation}\label{x5}
x(u,v)=e^{\frac{iu}{\sqrt{2}}}(e^{\frac{iv}{\sqrt{2}}},e^{\frac{-iv}{\sqrt{2}}})\in\mathbb{C}^2.
\end{equation}
We choose local coordinates $s,t$ such that $s=\frac{u+v}{\sqrt{2}},t=\frac{u-v}{\sqrt{2}}$, then $x$ is congruent with
\begin{equation}\label{x6}
x(s,t)=(e^{is},e^{it})\in\mathbb{C}^2,
\end{equation}
which is the standard expression of the Clifford torus $\mathbb{S}^1(1)\times\mathbb{S}^1(1)$ in $\mathbb{C}^2$.
\end{proof}

\noindent \textbf{Proof of Theorem \ref{thm1.2}~:}
Under the assumptions of Theorem \ref{thm1.2}, from Lemma \ref{lem3.3}, we know that $|A|^2=2$,  which means that $|A|^2$ is constant. Then applying our key Proposition \ref{prop3.5},   we obtain that  $M^2$ is the Clifford torus $\mathbb{S}^1(1)\times\mathbb{S}^1(1)$.
\qed

\begin{rem}
If $x:M^2\to\mathbb{R}^3$ is a compact orientable embedded self-shrinker with $|A|^2\leq 2$, then it follows from Gauss equation and Gauss-Bonnet theorem that
\begin{equation*}8\pi(1-\text{gen}(M^2))=2\int_M^2K dv=\int_M^2(H^2-|A|^2)dv=\int_M^2(2-|A|^2)dv,\end{equation*}
where $\text{gen}(M^2)$ stands for the genus of $M^2$, $K$ is the Gauss curvature of $M^2$ and the last equality is due to the following identity by using the self-shrinker equation \eqref{1.1}: $$\frac{1}{2}\Delta (|x|^2)=2+\l x,\Delta x\r=2+\l x,\mathbf{H}\r=2-H^2.$$
From $|A|^2\leq 2$ it follows  that either (i) the genus of $M^2$ is 0, or (ii) the genus of $M^2$ is 1 and $|A|^2=2$. If the genus of $M^2$ is 0, then Brendle's result (see \cite{Brendle}, Theorem 1)
implies that $M^2$ is the round sphere $\mathbb{S}^2(\sqrt{2})$.
If $M^2$ is a $2$-dimensional closed self-shrinker in $\mathbb{R}^3$ satisfying that $|A|^2$ is constant, then Ding and Xin's result (see Theorem 4.2 of \cite{DX}, see also \cite{Guang} for a new proof)  implies that  $|A|^2=1$.
Therefore, case (ii) the genus of $M^2$ is 1 and $|A|^2=2$ can not occur. So we obtain the following new characterization of the round sphere as a self-shrinker.
\end{rem}
\begin{prop}
Let $x:M^2\to\mathbb{R}^3$ be a compact orientable embedded self-shrinker. If $|A|^2\leq 2$, then $|A|^2=1$ and $M^2$ is the round sphere $\mathbb{S}^2(\sqrt{2})$.
\end{prop}
\section{Proof of Theorem \ref{thm1.3}}
In this section, we prove Theorem \ref{thm1.3}. We also prove that a compact orientable Lagrangian self-shrinker in $\mathbb{C}^2$ with constant Gauss curvature  must be the Riemannian product of two closed  Abresch-Langer curves.
As an application, we obtain several new characterizations of the Clifford torus as a Lagrangian self-shrinker in $\mathbb{C}^2$.

\begin{lem}\label{lem5.1}
Let $x:M^2\to \mathbb{C}^2$ be a compact orientable Lagrangian self-shrinker. If the Gauss curvature $K$ of $M^2$ is nonnegative, then $K=0$ and $M^2$ is a topological torus.
\end{lem}
\begin{proof}
Using the fact that there exist no Lagrangian self-shrinkers in $\mathbb{C}^n$ with the topology of sphere,which
was proved by Smoczyk (see \cite{Smo00}, Theorem 2.3.5, see also Theorem 2.1 in \cite{Castro} for a detailed proof), we get $\text{gen}(M^2)\geq 1$, where $\text{gen}(M^2)$ stands for the genus of $M^2$.
From Gauss-Bonnet theorem, we derive
\begin{equation}4\pi(1-\text{gen}(M^2))=\int_M^2K dv.\end{equation}
If $K\geq 0$,  then $\text{gen}(M^2)\leq 1$.  Hence, if  $K\geq 0$, then $K=0$ and $M^2$ is a topological torus.
\end{proof}

\begin{prop}\label{prop5.2}
Let $x:M^2\to \mathbb{C}^2$ be a compact orientable Lagrangian self-shrinker. If the Gauss curvature $K$ of $M^2$ is constant, then $K=0$ and $M^2$ is a topological torus.
\end{prop}
\begin{proof}
First, using Lemma \ref{lem5.1}, if $K\geq 0$, then $K=0$. Hence, if $K$ is constant,  then $K\leq 0$.
Next, we prove that $K=0$.
It follows from \eqref{3.3} and  $
\mathcal{L}=\Delta-\l x, \nabla\cdot\r
$that
\begin{equation}\label{DeltaH}
\begin{aligned}
\frac{1}{2}\mathcal{L} |\mathbf{H}|^2&=\frac{1}{2}\mathcal{L} H^2=\sum_{k,i}(H^{k^*}_{,i})^2+H^2-\sum_{i,j,k,l}H^{k^*}H^{l^*}h_{ij}^{k^*}h_{ij}^{l^*}\\
&=|\nabla^{\bot}\mathbf{H}|^2+H^2-\sum_{i,j,k,l}H^{k^*}H^{l^*}h_{ij}^{k^*}h_{ij}^{l^*},
\end{aligned}
\end{equation}
on the other hand, from Lemma \ref{lem3.1} we know that
\begin{equation}\label{DeltaA2}
\begin{aligned}
\frac{1}{2}\mathcal{L} |A|^2=|\nabla A|^2+|A|^2-\frac{3}{2}|A|^4+2H^2|A|^2-\frac{1}{2}H^4-\sum_{i,j,k,l}H^{k^*}H^{l^*}h_{ij}^{k^*}h_{ij}^{l^*}.
\end{aligned}
\end{equation}
Therefore, using Gauss equation $2K=H^2-|A|^2$, we derive
\begin{equation}\label{DeltaK}
\begin{aligned}
\mathcal{L} K=|\nabla^{\bot}\mathbf{H}|^2-|\nabla A|^2+H^2-(|A|^2-\frac{3}{2}|A|^4+2H^2|A|^2-\frac{1}{2}H^4).
\end{aligned}
\end{equation}

As $M^2$ is compact, there exists a point $p_0\in M^2$ such that  $|x|^2$ attains its maximum at $p_0$.
We immediately have $(|x|^2)_{,j}=0,~1\leq j\leq 2$ at $p_0$, which implies that $\l x,e_j\r(p_0)=0,~1\leq j\leq 2$.
Hence at $p_0$, from \eqref{3.1} and \eqref{3.2} we have $x=-\mathbf{H},~|x|^2=H^2$, $H^{k^*}_{,i}=0,~1\leq i,k\leq 2$.
On the other hand, since $K$ is constant, $K_{,k}\equiv0,~1\leq k\leq 2$. Using Gauss equation  $2K=H^2-|A|^2$, we
get that $(|A|^2)_{,k}=0,~1\leq k\leq 2$ at $p_0$. Hence, \eqref{3.7} and \eqref{3.8} hold at $p_0$. Using the same argument as in the proof of
Proposition \ref{prop3.5}, we have the following two possibilities:

(i) At $p_0$, $h_{11}^{1^*}=3h_{22}^{1^*},~h_{22}^{2^*}=3h_{11}^{2^*}$. In this case, $|A|^2(p_0)=\frac{4}{3}((h_{11}^{1^*})^2+(h_{22}^{2^*})^2),H^2(p_0)=\frac{16}{9}((h_{11}^{1^*})^2+(h_{22}^{2^*})^2)$,
so we get $|A|^2(p_0)=\frac{3}{4}H^2(p_0)$, $0\geq K=K(p_0)=\frac{1}{2}(H^2(p_0)-|A|^2(p_0))=\frac{1}{8}H^2(p_0)\geq0$, which implies that
$K=0$.

(ii) At $p_0$,  $h_{11,1}^{1^*}=h_{11,2}^{1^*}=h_{22,1}^{1^*}=h_{22,2}^{1^*}=h_{22,2}^{2^*}=0$.
In this case, $|\nabla A|=0$ at $p_0$. Since $K$ is constant, we get $\mathcal{L} K\equiv 0$, then at $p_0$, from \eqref{DeltaK} we have
\begin{equation}
H^2-(|A|^2-\frac{3}{2}|A|^4+2H^2|A|^2-\frac{1}{2}H^4)=\frac{1}{2}(H^2-3|A|^2+2)(H^2-|A|^2)=0,
\end{equation}
from which we deduce that either $H^2(p_0)=3|A|^2(p_0)-2$ or $H^2(p_0)=|A|^2(p_0)$.
If $H^2(p_0)=3|A|^2(p_0)-2$, then $K=K(p_0)=\frac{1}{2}(H^2(p_0)-|A|^2(p_0))=\frac{1}{3}(H^2(p_0)-1)\leq 0$, so we get $H^2(p_0)\leq 1$.
On the other hand,  since $|x|^2$ attains its maximum at $p_0$, from Lemma \ref{lem3.2},
we deduce that
$H^2(p_0)=|x|^2(p_0)\geq 2$, which contradicts with $H^2(p_0)\leq 1$. So we get $H^2(p_0)=|A|^2(p_0)$, which implies that $K=K(p_0)=0$.

Therefore, we have proved that $K=0$. It follows from Gauss-Bonnet theorem that $M^2$ is a topological torus.
\end{proof}

\begin{prop}\label{prop5.3}
Let $x:M^2\to \mathbb{C}^2$ be a Lagrangian self-shrinker. If $M^2$ is flat, then $M^2$ is locally an open part of the Riemannian product of two Abresch-Langer curves.
\end{prop}
\begin{proof}
We define $U_1=\{p\in M^2~|p ~~\text{is a totally geodesic point}\}$, $U_2=M^2-U_1$.
If $p\in M^2$ is an interior point of  $U_1$, then  $M^2$ is locally the Riemannian product of two straight lines around $p$ (straight line is also a special Abresch-Langer curve).

In the following, without loss of generality,
we assume that $p\in U_2$, i.e., $p$ is not a totally geodesic point.
We denote $UM_p=\{u\in T_pM^2| ~|u|=1\}$, define $f(u)=\l h(u,u),Ju\r (u\in UM_p)$ and take  $e_1$ as a vector in which $f$ attains its maximum.
We choose $e_2\in T_pM^2$ as a unit vector which is orthogonal to $e_1$. As $f$ attains its maximum in $e_1$, we immediately have
$\l h(e_1,e_1),Je_2\r=0$, which implies that there exists  a number $\lambda_1>0$ such that $h(e_1,e_1)=\lambda_1Je_1$.
Since $\l h(X,Y),JZ\r$ is totally symmetric  (see \eqref{2.1}), there exist two numbers $\lambda_0$ and $\lambda_2$ such that $$h(e_1,e_2)=\lambda_0 Je_2,~h(e_2,e_2)=\lambda_0Je_1+\lambda_2Je_2.$$ Moreover, since $f$ attains its maximum in $e_1$, we have $\lambda_1\geq 2\lambda_0$, and if $\lambda_1=2\lambda_0$, then $\lambda_2=0$ (see Lemma 1 in \cite{MU}).
As $M^2$ is flat, from Gauss equation we have $$0=\l h(e_1,e_1),h(e_2,e_2)\r-\l h(e_1,e_2),h(e_1,e_2)\r=\lambda_0(\lambda_1-\lambda_0).$$
We claim that $\lambda_0=0$, if not,  $0<\lambda_1=\lambda_0$ which contradicts with $\lambda_1\geq 2\lambda_0$. So we obtain an orthonormal basis $e_1,e_2$ at $p$ such that
\begin{equation}
h(e_1,e_1)=\lambda_1Je_1,~h(e_1,e_2)=0,~h(e_2,e_2)=\lambda_2Je_2.
\end{equation}

Next, we prove that there exists a neighborhood $U$ of $p$, local orthonormal vector fields $E_1,E_2$  and local functions $\Lambda_1,\Lambda_2$ such that at each point $q\in U$, we have
\begin{equation}\label{frame}
h(E_1(q),E_1(q))=\Lambda_1(q)JE_1(q),~h(E_1(q),E_2(q))=0,~h(E_2(q),E_2(q))=\Lambda_2(q)JE_2(q).
\end{equation}
We choose an arbitrary orthonormal vector field $F_1,F_2$ in a neighborhood $V$ of $p$ such that $F_i(p)=e_i$. We denote $h_{ij}^k(q)=\l h(F_i(q),F_j(q)),JF_k(q)\r~(\forall~ q\in V)$  and consider the following system of equations:
\begin{equation}\label{l}
\left\{
\begin{aligned}
L_1(y^1(q),y^2(q),\Lambda_1(q))&:=\sum_{j,k}h_{jk}^1(q)y^j(q)y^k(q)-y^1(q)\Lambda_1(q)=0,\\
L_2(y^1(q),y^2(q),\Lambda_1(q))&:=\sum_{j,k}h_{jk}^2(q)y^j(q)y^k(q)-y^2(q)\Lambda_1(q)=0,\\
L_{3}(y^1(q),y^2(q),\Lambda_1(q))&:=(y^1(q))^2+(y^2(q))^2-1=0.
\end{aligned}
\right.
\end{equation}
If we denote $L=(L_1,L_2,L_3),Y=(y^1,y^2,\Lambda_1)$, then  $Y(p)=(1,0,\lambda_1)$ is a solution to $L(p)=0$,
and
\begin{equation}\label{h}
\left(\tfrac{\partial L_i}{\partial Y_j}\right)\Big|_p
=\left(
\begin{array}{cccc}
\lambda_1&0&-1\\
0&-\lambda_1&0\\
2&0&0
\end{array}
\right)
\end{equation}
is non-degenerate. Applying Implicit Function Theorem, there exists a unique smooth function $Y(q)=(y^1(q),y^2(q),\Lambda_1(q))$
satisfying \eqref{l} in an open set $V_1\subset V$, with initial value $Y(p)=(1,0,\lambda_1)$. If we define $E_1(q)=\sum_{i=1}^2y^i(q)F_i(q)$, then \eqref{l} implies that $E_1$  is a smooth unit vector field in $V_1$  and $h(E_1(q),E_1(q))=\Lambda_1(q)JE_1(q),~\forall ~q\in V_1$.

Assume that $E_2$ is a smooth unit vector field in $V_1$ such that $E_1$ and $E_2$ are orthogonal,
using the property that $\l h(X,Y),JZ\r$ is totally symmetric  (see \eqref{2.1}), we get that there exist two local functions $\Lambda_0$ and $\Lambda_2$ such that $$h(E_1,E_2)=\Lambda_0 JE_2,~h(E_2,E_2)=\Lambda_0 JE_1+\Lambda_2JE_2,~\forall ~q\in V_1.$$
As $(\Lambda_1-\Lambda_0)(p)=\lambda_1-\lambda_0=\lambda_1-0>0$, there exists an open set $U\subset V_1$ such that $$(\Lambda_1-\Lambda_0)(q)>0,~\forall~ q\in U.$$
Moreover, since $M^2$ is flat, from Gauss equation we have $$0=\l h(E_1,E_1),h(E_2,E_2)\r-\l h(E_1,E_2),h(E_1,E_2)\r=\Lambda_0(\Lambda_1-\Lambda_0),$$
so we derive $\Lambda_0=0,~\forall ~q\in U$. Therefore, we have found  a neighborhood $U$ of $p$, local orthonormal vector fields $E_1,E_2$  and local functions $\Lambda_1,\Lambda_2$ such that at each point $q\in U$, \eqref{frame} is satisfied.

In the following, we  use Codazzi equations and the self-shrinker equation to deduce that $x$ is locally a  product immersion.
As $E_1$ and $E_2$ are local orthonormal tangent vector fields, we can write the covariant derivatives as follows.
\begin{equation}\label{conn}
\nabla_{E_1}E_1=\alpha E_2,~\nabla_{E_1}E_2=-\alpha E_1,~\nabla_{E_2}E_1=-\beta E_2,~\nabla_{E_2}E_2=\beta E_1,
\end{equation}
where $\alpha$ and $\beta$ are local functions.
It follows from \eqref{frame}, \eqref{conn} and the Codazzi equation
$(\nabla_{E_1} h)(E_2,E_2)=(\nabla_{E_2} h)(E_1,E_2)$ that
\begin{equation}\label{codazzi1}
\alpha\Lambda_2-\beta\Lambda_1=0,
\end{equation}
\begin{equation}\label{codazzi3}
E_1(\Lambda_2)=\beta\Lambda_2.
\end{equation}
If we denote $x^{T}=x-x^{\bot}$, then $x^{T}$ is the tangent part of the position vector $x$. By using \eqref{GW}, \eqref{kahler},  \eqref{2.1}, \eqref{frame} and \eqref{conn}, we derive
 \begin{equation*}
\left\{
\begin{aligned}
&\l D_{E_1}(-x^{\bot}),JE_2\r=\l D_{E_1}(x^{T}-x),JE_2\r=\l D_{E_1}(x^{T})-D_{E_1}x,JE_2\r\\
&=\l
 D_{E_1}(x^{T})-E_1,JE_2\r=\l D_{E_1}(x^{T}),JE_2\r=\l
\nabla_{E_1}(x^{T})+h(E_1,x^{T}),JE_2\r\\
&=\l h(E_1,x^{T}),JE_2\r=\l h(E_1,E_2),Jx^{T}\r=\l 0,Jx^{T}\r=0,\\
&\l D_{E_1}\mathbf{H},JE_2\r=\l\nabla^{\bot}_{E_1}\mathbf{H},JE_2\r=\l\nabla^{\bot}_{E_1}(\Lambda_1JE_1+\Lambda_2JE_2),JE_2\r=E_1(\Lambda_2)+\alpha\Lambda_1,
\end{aligned}
\right.
\end{equation*}
 which combined with the self-shrinker equation \eqref{1.1} ($\mathbf{H}=-x^{\bot}$) imply
\begin{equation}\label{shrinker}
\begin{aligned}
E_1(\Lambda_2)+\alpha\Lambda_1=0.
\end{aligned}
\end{equation}
From \eqref{codazzi3} and \eqref{shrinker} we derive
\begin{equation}\label{cs}
\alpha \Lambda_1+\beta\Lambda_2=0.
\end{equation}
Taking the sum of the square of \eqref{codazzi1} and the square of \eqref{cs}, we derive
\begin{equation}
(\alpha^2+\beta^2)(\Lambda_1^2+\Lambda_2^2)=0,
\end{equation}
since $\Lambda_1>0$ on $U$, we conclude that $\alpha=\beta=0$ on $U$, which means that $E_1$ and $E_2$ are both totally geodesic distributions on $U$.
Therefore, applying the theorem of Frobenius, there exist local coordinates $\{s,t\}$ on $U$ such that $E_1=\frac{\partial}{\partial s},~E_2=\frac{\partial}{\partial t}$, and $M^2$ is locally a Riemannian product $I_1\times I_2\in \mathbb{R} \times \mathbb{R}$. Since the second fundamental form satisfies \eqref{frame}, using a lemma of J. D. Moore (see Lemma in the end of section 2 of \cite{Moore}), we know that $x$ is locally a product immersion. Here we present a direct proof of this conclusion. Since $E_1=\frac{\partial}{\partial s}$ and $E_2=\frac{\partial}{\partial t}$ are orthonormal, $x$ is a Lagrangian immersion, we derive
\begin{equation}\label{xst}
\l x_s,x_s\r=\l x_t,x_t\r=1,~\l x_s,x_t\r=\l x_s,i x_t\r=0.
\end{equation}
From \eqref{frame} and  \eqref{conn}, using $\alpha=\beta=0$,
we have
\begin{equation}\label{xsstt}
\begin{aligned}
x_{st}=0,
\end{aligned}
\end{equation}
which implies that there exist four complex functions $f_i(s),g_i(t),i=1,2$ such that
\begin{equation}\label{expressx}
x=(f_1(s)+g_1(t),f_2(s)+g_2(t))\in\mathbb{C}^2.
\end{equation}
\eqref{expressx} combined with \eqref{xst} imply
\begin{equation}\label{fg}
\left\{
\begin{aligned}
&|f_1'(s)|^2+|f_2'(s)|^2=1,~|g_1'(t)|^2+|g_2'(t)|^2=1,\\
&|f_1'(s)|^2+|g_1'(t)|^2=1,~|f_2'(s)|^2+|g_2'(t)|^2=1,\\
&f_1'(s)\bar{g_1}'(t)+f_2'(s)\bar{g_2}'(t)=0,\\
&f_1'(s)\bar{f_2}'(s)+g_1'(t)\bar{g_2}'(t)=0.
\end{aligned}
\right.
\end{equation}
In \eqref{fg}, the  equations  in the first and third lines are direct consequences of \eqref{expressx} combined with \eqref{xst} and  these equations mean that the matrix $
A=\left[
\begin{array}{cc}
f_1'(s)&f_2'(s)\\
g_1'(t)&g_2'(t)
\end{array}
\right]
$
is a unitary matrix, so we obtain the equations in the second and forth lines of \eqref{fg}.

Using \eqref{fg}, there exist two real constants $\theta_0,~\theta_1$ and two real functions $f(s),g(t)$ such that
\begin{equation*}
f_1'(s)=\cos{\theta_0}e^{if(s)},~f_2'(s)=\sin{\theta_0}e^{i\theta_1}e^{if(s)},~g_1'(t)=-\sin{\theta_0}e^{ig(t)},~g_2'(t)=\cos{\theta_0}e^{i\theta_1}e^{ig(t)}.
\end{equation*}
If we denote $F(s)=\int_0^s e^{if(\tilde{s})}d\tilde{s},~G(t)=\int_0^t e^{ig(\tilde{t})}d\tilde{t}$, then we obtain
\begin{equation}
x=(\cos{\theta_0}F(s)-\sin{\theta_0}G(t)+c_1,\sin{\theta_0}e^{i\theta_1}F(s)+\cos{\theta_0}e^{i\theta_1}G(t)+c_2),
\end{equation}
where $c_1$ and $c_2$ are two complex constants. By solving $$\cos{\theta_0}a_1-\sin{\theta_0}a_2=c_1,~\sin{\theta_0}e^{i\theta_1}a_1+\cos{\theta_0}e^{i\theta_1}a_2=c_2,$$
we get a unique solution for $a_1$ and $a_2$, so $x$ can be expressed as
\begin{equation}
x=(\cos{\theta_0}(F(s)+a_1)-\sin{\theta_0}(G(t)+a_2),\sin{\theta_0}e^{i\theta_1}(F(s)+a_1)+\cos{\theta_0}e^{i\theta_1}(G(t)+a_2)),
\end{equation}
where $a_1$ and $a_2$ are two complex constants.
Therefore, up to an isometry of $\mathbb{C}^2$, $x$ is locally congruent with
\begin{equation}
x(s,t)=(x_1(s),x_2(t))=(F(s)+a_1,G(t)+a_2)\in\mathbb{C}^2,
\end{equation}
which is locally a  product immersion from a Riemannian product $I_1\times I_2$ to $\mathbb{C}^2$.
Finally, since $M^2$ is a self-shrinker, from the self-shrinker equation \eqref{1.1}, we obtain that $x_1(s):I_1\to\mathbb{C}$ and $x_2(t):I_2\to\mathbb{C}$ also satisfy the self-shrinker equation \eqref{1.1}, hence we obtain that
$M^2$ is locally an open part of the Riemannian product of two one-dimensional self-shrinkers in $\mathbb{C}=\mathbb{R}^2$, i.e.,
$M^2$ is locally an open part of  the Riemannian product of two Abresch-Langer curves.
\end{proof}

There is a special property of the Abresch-Langer curves (see Theorem A in \cite{AL} and Lemma 5.3 in \cite{Smo05}):
\begin{lem}[see  Lemma 5.3 in \cite{Smo05}]\label{lem5.4}
If $x:\Gamma\to\mathbb{R}^2$ is an Abresch-Langer curve, $k$ is the curvature of $\Gamma$ with respect to its inner unit normal, then there exists a constant $c_{\Gamma}$ such that
$ke^{-|x|^2/2}=c_{\Gamma}$ holds on all of $\Gamma$. If $\Gamma_1,~\Gamma_2$ are two Abresch-Langer  curves with $c_{\Gamma_1}=c_{\Gamma_2}$,
then up to a Euclidean motion $\Gamma_1=\Gamma_2$. Moreover, $k_{\text{min}}$ and $k_{\text{max}}$ satisfy
$k_{\text{min}}e^{-k_{\text{min}}^2/2}=k_{\text{max}}e^{-k_{\text{max}}^2/2}=c_{\Gamma}$, hence $k_{\text{min}}(\Gamma)>0$ if $\Gamma$ is not a  straight line through the origin.
\end{lem}

Applying Lemma \ref{lem5.4}, we know that if two Abresch-Langer curves $\Gamma_1$ and $\Gamma_2$ coincide on an open set, then $\Gamma_1$ and $\Gamma_2$
coincide completely. Consequently, if the Riemannian product of two Abresch-Langer curves ($x:\Gamma_1\times \Gamma_2\to\mathbb{C}^2$) and the Riemannian product of other two Abresch-Langer curves ($\tilde{x}:\tilde{\Gamma_1}\times \tilde{\Gamma_2}\to\mathbb{C}^2$) coincide on an open set, then $\Gamma_1\times \Gamma_2$  and $\tilde{\Gamma_1}\times \tilde{\Gamma_2}$ coincide completely.
Using Lemma \ref{lem5.4} and Proposition \ref{prop5.3}, we conclude that Proposition \ref{prop5.3} is also true
in the global sense.
\begin{prop}\label{prop5.5}
Let $x:M^2\to \mathbb{C}^2$ be a complete connected Lagrangian self-shrinker. If $M^2$ is flat, then $M^2$ is  the Riemannian product of two Abresch-Langer curves.
\end{prop}
\begin{proof}
We use the same notations as in Proposition \ref{prop5.3}.

 We define $U_1=\{p\in M^2~|p ~~\text{is a totally geodesic point}\}$, $U_2=M^2-U_1$.
It is obvious that $U_1$ is a closed set and $U_2$ is an open set.
We prove that either $M^2=U_1$ or $M^2=U_2$. As $U_2$ is an open set, we immediately get that
$U_2=\cup_k U_{2k}$, where $U_{2k}$ are open disjoint connected components of $U_2$.
For any $k$, $\forall~ p\in U_{2k}$, by using Proposition \ref{prop5.3}, we know that there exists a neighborhood $U_p\subset U_{2k}$ such that  $U_p$ is an open part of the Riemannian product of
two Abresch-Langer curves. We denote $U_{2k}=\cup_{p\in U_{2k}}U_p$.  If $p_1,p_2\in U_{2k},~U_{p_1}\cap U_{p_2}\neq \emptyset$,  $U_{p_1}$ is an open part of $\Gamma_1\times \Gamma_2$, $U_{p_2}$ is an open part of $\tilde{\Gamma_1}\times \tilde{\Gamma_2}$, then $\Gamma_1\times \Gamma_2$ and $\tilde{\Gamma_1}\times \tilde{\Gamma_2}$ coincide on the nonempty open set $U_{p_1}\cap U_{p_2}$, so we obtain that $\Gamma_1\times \Gamma_2$  and $\tilde{\Gamma_1}\times \tilde{\Gamma_2}$ coincide completely. This implies that each component $U_{2k}$ is an open part of two Abresch-Langer curves $\Gamma_{1k}\times \Gamma_{2k}$. By definition of $U_2$, $\forall~ p\in U_{2k}$, $p$ is not a totally geodesic point, without loss of generality, we assume that $\Gamma_{1k}$ is not a straight line, then we get that $k(\Gamma_{1k})\geq k_{\text{min}}(\Gamma_{1k})>0$, then by use of continuity, $\forall ~\tilde{p} \in \overline{U_{2k}}$, $|A|^2(\tilde{p} )\geq k^2_{\text{min}}(\Gamma_{1k})>0$, so we deduce that $\tilde{p} $ is not a totally geodesic point, which means that $\overline{U_{2k}}\subset U_2$. On the other hand, $U_{2k}$ is a connected component of $U_2$,
so we get $\overline{U_{2k}}=U_{2k}$, hence $U_{2k}$ is open and closed, which implies that either $U_{2k}=\emptyset$ or $U_{2k}=M^2$. Therefore, there are two possibilities: (i) $U_{2k}=\emptyset,~\forall~ k$. In this case, $U_2=\emptyset$ and $M^2=U_1$; (ii) $\exists ~k $ s.t. $U_{2k}=M^2$. In this case $U_1=\emptyset$ and $M^2=U_2$.

If $M^2=U_1$, then $M^2$ is totally geodesic. As $M^2$ is complete and connected,  we obtain that $M^2$ is the Riemannian product of two straight lines (straight line is also a special Abresch-Langer curve).

If $M^2=U_2$, since $M^2$ is complete and connected, then from the arguments above,  we get that $M^2$ is  the Riemannian product of two Abresch-Langer curves.
\end{proof}

\noindent \textbf{Proof of Theorem \ref{thm1.3}~:}
Under the assumptions of Theorem \ref{thm1.3}, from Lemma \ref{lem5.1}, we know that $K=0$ and $M^2$ is a topological torus. Then applying Proposition \ref{prop5.5}, we obtain that $M^2$ is  the Riemannian product of two Abresch-Langer curves.
\qed

Combing Proposition \ref{prop5.2} and Proposition \ref{prop5.5}, we obtain
\begin{prop}\label{prop5.6}
Let $x:M^2\to \mathbb{C}^2$ be a compact orientable Lagrangian self-shrinker. If the Gauss curvature $K$ of $M^2$ is constant, then $K=0$ and $M^2$ is the Riemannian product of two closed Abresch-Langer curves.
\end{prop}

If $x:M^2\to \mathbb{C}^2$ is embedded, using the result of Abresch-Langer which states that
the only closed embedded self-shrinker in $\mathbb{R}^2$ is the circle, as immediate consequences of Proposition \ref{prop5.6}, we obtain

\begin{cor}\label{newcor5.6}
 The Clifford torus $\mathbb{S}^1(1)\times\mathbb{S}^1(1)$  is the unique compact orientable embedded  Lagrangian self-shrinker in $\mathbb{C}^2$ with constant Gauss curvature.
\end{cor}
\begin{rem}
By a theorem of Whitney, any compact (without boundary) orientable embedded Lagrangian surface $M^2$ in $\mathbb{C}^2$ has to be a topological torus. If we assume that $M^2$ is a compact orientable embedded Lagrangian self-shrinker in $\mathbb{C}^2$ with nonpositive Gauss curvature, then by using Gauss-Bonnet theorem we obtain $K=0$.
Using Proposition \ref{prop5.5} and the result of Abresch-Langer which states that
the only closed embedded self-shrinker in $\mathbb{R}^2$ is the circle, we deduce the following characterization of the Clifford torus.
\end{rem}
\begin{cor} \label{cor5.10}
 The Clifford torus $\mathbb{S}^1(1)\times\mathbb{S}^1(1)$  is the unique compact orientable embedded  Lagrangian self-shrinker in $\mathbb{C}^2$ with nonpositive Gauss curvature.
\end{cor}

\bigskip
\noindent {\bf Acknowledgements}: The authors would like to thank the referee  for the very careful review and for providing a number of  valuable comments and suggestions.

\end{document}